\documentclass{article}
\usepackage[T1]{fontenc}
\usepackage[margin=1.2in]{geometry}
\usepackage{graphicx} % Required for inserting images
\usepackage{amsthm}   % For theorem environments
\usepackage{amsmath}  % For math environments
\usepackage{color}    % For comments
\usepackage{hyperref}
\usepackage{amsfonts}
\usepackage{bm}
\usepackage{subcaption}
\usepackage{svg}
\usepackage{appendix}
\usepackage{mathtools}
\usepackage{authblk}
\usepackage[
backend=biber,
style=alphabetic,
url=false,
doi=false,
isbn=false
]{biblatex}
\renewbibmacro{in:}{}

\title{Shellability of relative squeezed balls and spheres}
\author{Luz Elena Grisales Gómez}
\affil{Department of Mathematics\\
University of Washington\\
Seattle, WA 98195-4350, USA}
\date{October 2026}

\numberwithin{equation}{section}

\theoremstyle{plain}
\newtheorem{theorem}{Theorem}[section]
\newtheorem{proposition}[theorem]{Proposition}
\newtheorem{propositionDefinition}[theorem]{Proposition and Definition}
\newtheorem{lemma}[theorem]{Lemma}
\newtheorem{corollary}[theorem]{Corollary}

\newenvironment{numlemma}[1]{\medskip\noindent {\bf Lemma #1.}\ \em}{}

\theoremstyle{remark}
\newtheorem{remark}[theorem]{Remark}

\theoremstyle{definition}
\newtheorem{notation}[theorem]{Notation}
\newtheorem{definition}[theorem]{Definition}

\begin{document}

\maketitle

\begin{abstract}
    Squeezed balls and spheres, introduced by Kalai, form a rich class of triangulated complexes arising from subcomplexes of cyclic polytopes, with well-understood shellability properties. Recently, Novik and Zheng introduced relative squeezed balls, obtained as differences of squeezed balls, and used them to construct large families of highly neighborly simplicial spheres. While these complexes are known to be constructible, their shellability has remained open. In this paper, we resolve this question by proving that both relative squeezed balls and their boundary complexes are shellable. We provide explicit shelling orders and characterize restriction faces, thereby establishing strong combinatorial structure for this new class of complexes.

\end{abstract}

\section{Introduction}
Squeezed balls and spheres were introduced by Kalai \cite{kalaiManyTriangulatedSpheres1988} to obtain a large family of triangulated $(d-1)$-spheres with $n$ vertices. They arise as specific subcomplexes of the boundary complex of the cyclic polytope $C_{d+1}(n)$ and can be viewed as a generalization of the Billera--Lee construction \cite{billeraProofSufficiencyMcMullens1981}, which was used to prove the sufficiency of McMullen's conditions in the $g$-theorem. The boundary complex of a squeezed ball is called a \emph{squeezed sphere}. For these simplicial complexes, shellability is well understood. Kalai proved that squeezed balls are shellable, and Lee \cite{leeKalaisSqueezedSpheres2000} later established shellability of squeezed spheres.

Relative squeezed balls were later introduced by Novik and Zheng \cite{novikManyNeighborlySpheres2024} upon observing that, under suitable conditions, the difference\footnote{Given two simplicial complexes $\Delta_1$ and $\Delta_2$, their `difference', denoted $\Delta_1 \setminus \Delta_2$, refers to the simplicial complex whose facets belong to $\Delta_1$ but not to $\Delta_2$.} of two squeezed balls results in a ball. This resulting ball is called a \emph{relative squeezed ball}, and its boundary complex is called a \emph{relative squeezed sphere}. Novik and Zheng used these complexes to construct a large family of $\lfloor d/2 \rfloor$-neighborly simplicial $(d-1)$-spheres with $n$ labeled vertices. While constructibility of relative squeezed balls is known \cite{novikManyNeighborlySpheres2024}, whether relative squeezed balls and their boundaries are shellable has remained open.

The importance of resolving this question stems from the strong structural consequences of shellability. In particular, shellability endows a complex with a rich combinatorial framework that is particularly well suited for inductive arguments. Moreover, shellability often leads to proofs that are substantially more elementary or transparent than those available in the general setting; for instance, by exploiting shellability, Alon and Kalai obtained a proof of the Upper Bound Theorem for shellable simplicial spheres \cite{alonSimpleProofUpper1985} that only relies on linear algebra tools, and as such is more elementary than Stanley’s original proof for arbitrary simplicial spheres \cite{stanleyUpperBoundConjecture1975}. Finally, several important properties are known to hold for shellable spheres but remain open for general spheres; for example, shellable spheres can be reconstructed from their facet–ridge graph \cite{yangReconstructingShellableSphere2024}, whereas whether this reconstruction property holds in full generality is still unresolved.

In this paper, we provide a positive answer to the shellability question for relative squeezed balls and relative squeezed spheres. The structure of the paper is as follows: In Section \ref{sec:preliminaries} we introduce several preliminary definitions, constructions, and results that will be used throughout the rest of the paper. In Section \ref{sec:spheres} we prove that relative squeezed spheres are shellable (see Theorem \ref{thm:relative_squeezed_sphere_shellable}) using a generalization of Lee's shelling for squeezed spheres \cite{leeKalaisSqueezedSpheres2000}, and standard shellability results. Finally, in Section \ref{sec:balls} we prove that relative squeezed balls (see Theorem \ref{thm:relative_squeezed_ball_shellable}) are shellable by providing a shelling order and a description of the restriction faces. 

\section{Preliminaries}\label{sec:preliminaries}

In this section we review several definitions, constructions, and results that will be used throughout the paper. Our exposition follows standard terminology in the study of simplicial complexes; for additional background and further details, the reader is referred to Chapter~8 of Ziegler’s book \emph{Lectures on Polytopes} \cite{zieglerLecturesPolytopes2007}.

\subsection{Simplicial complexes}

A \emph{simplicial complex} $\Delta$ with vertex set $V(\Delta)$ is a collection of subsets of $V(\Delta)$ that is closed under inclusion and contains all singletons: $\{v\} \in \Delta$ for every $v \in V(\Delta)$. The elements of $\Delta$ are called \emph{faces}. 

For any collection $\mathcal{F}$ of faces of a simplicial complex, we denote by 
$\overline{\mathcal{F}}$ the simplicial complex generated by $\mathcal{F}$; that is,
\[
\overline{\mathcal{F}} \coloneq \{\tau : \tau \subseteq F \text{ for some } F \in \mathcal{F}\}.
\]
In particular, for a single face $F$ we write 
$\overline{F} \coloneq \overline{\{F\}}$, the simplex on the vertex set $F$. 

The \emph{dimension} of a face $F \in \Delta$ is  $\dim F \coloneq |F|-1$, and the dimension of $\Delta$ is the maximum dimension of any face of $\Delta$. A face that is maximal under inclusion is called a \emph{facet}. We say that $\Delta$ is \emph{pure} if all of its facets have the same dimension.

The \emph{geometric realization} $|\Delta|$ of a simplicial complex $\Delta$ is obtained by identifying each face $\tau = \{v_{i_0},\dots,v_{i_k}\}$ with a geometric $k$-simplex and gluing these simplices together along their common subfaces in the natural way. A pure simplicial complex $\Delta$ is called a \emph{simplicial $d$-ball} (respectively, a \emph{simplicial $d$-sphere}) if its geometric realization $|\Delta|$ is homeomorphic to the $d$-dimensional ball $B^d$ (respectively, to the $d$-dimensional sphere $S^d$). If $\Delta$ is a simplicial $d$-ball, then the \emph{boundary complex} of $\Delta$, denoted $\partial \Delta$, is the subcomplex of $\Delta$ generated by all $(d-1)$-faces that are contained in exactly one facet of $\Delta$. This definition implies $|\partial \Delta| = \partial |\Delta| \cong S^{d-1}$.

If $\Delta$ is a pure simplicial complex and $\Gamma$ is a full-dimensional pure subcomplex, then $\Delta \setminus \Gamma$ denotes the subcomplex generated by those facets of $\Delta$ that are not facets of $\Gamma$. Finally, if $\Delta$ and $\Gamma$ are simplicial complexes on disjoint vertex sets, their \emph{join} is the simplicial complex
\[
\Delta * \Gamma \coloneq \{\sigma \cup \tau : \sigma \in \Delta 
\text{ and } \tau \in \Gamma\}.
\]

\subsection{Constructibility, shellability, and restriction faces}

A pure simplicial complex $\Delta$ of dimension $d$ is \emph{constructible} if either $\Delta$ is a simplex, or $\Delta = \Delta_1 \cup \Delta_2$, where $\Delta_1$ and $\Delta_2$ are constructible complexes of the same dimension as $\Delta$, and $\Delta_1 \cap \Delta_2$ is constructible of dimension $d-1$.

Let $\Delta$ be a pure simplicial complex of dimension $d$. A \emph{shelling} of $\Delta$ is an ordering of the facets
$F_1,\dots,F_m$ such that for every $i>1$, the intersection $\overline{F_i} \cap \left( \overline{F_1} \cup \cdots \cup \overline{F_{i-1}} \right)$ is a pure $(d-1)$-dimensional complex. A complex that admits a shelling is called \emph{shellable}. 

It follows directly from the definitions that all shellable complexes are constructible. While currently no examples are known of constructible spheres that fail to be shellable, there do exist constructible balls that are not shellable; prominent examples, in chronological order, include Rudin’s 3-ball \cite{rudinUnshellableTriangulationTetrahedron1958}, Grünbaum’s 3-ball \cite{danarajKleeWhichSpheresAreShellable}, and Ziegler’s 3-ball \cite{zieglerShellingPolyhedral3Balls1998a}.

One effective method for proving that an ordering $F_1,\dots,F_m$ is a shelling is via \emph{restriction faces}.

\begin{definition}\label{def:restriction-face}
Let $F_1,\dots,F_m$ be an ordering of the facets of a pure simplicial complex $\Delta$.
For each $1 \leq j\leq m$, a face $X_j \subseteq F_j$ is called the \emph{restriction face} of $F_j$ if the following two conditions hold:
\begin{enumerate}
    \item For every $x \in X_j$, there exists some $i<j$ such that
    $F_j \setminus \{x\} \subseteq F_i.$
    \item For every $i<j$, one has $X_j \not\subseteq F_i$.
\end{enumerate}
\end{definition}

\begin{lemma}[Restriction Face Criterion\footnote{See Exercise 8.2 of \cite{zieglerLecturesPolytopes2007}.}]\label{lemma:restriction-face-criterion}
Let $F_1,\dots,F_m$ be an ordering of the facets of a pure simplicial complex $\Delta$. Each $F_j$ admits a restriction face $X_j$ if and only if $F_1,\dots,F_m$ is a shelling of $\Delta$.
\end{lemma}

The restriction face $X_j$ is also known as the \emph{unique minimal new face} added at step $j$. This is due to the fact that the collection of faces introduced at step $j$ is exactly the interval
\[
[X_j,F_j] = \{\, G : X_j \subseteq G \subseteq F_j \,\}
\]
in the face poset of $\Delta$. Another key property of restriction faces is that, given a shelling of $\Delta$, the $h$-vector of $\Delta$ can be read off directly from the restriction faces. Namely,
\[
h_i(\Delta) = \#\{\text{restriction faces of size } i\}.
\]

\subsection{The cyclic polytope}
Let $m : \mathbb{R} \to \mathbb{R}^d$, 
$t \mapsto (t, t^2, \dots, t^d)$, be the moment curve in $\mathbb{R}^d$, 
and let $t_1 < t_2 < \cdots < t_n$ be distinct real numbers, where $n > d$. 
The cyclic $d$-polytope $C_d(n)$ is defined as the convex hull
\[
C_d(n) \coloneq \operatorname{conv}(m(t_1), \dots, m(t_n)).
\]
It is known that $C_d(n)$ is a simplicial $d$-polytope with $n$ vertices, 
that it is $\lfloor d/2 \rfloor$-neighborly, and that its combinatorial type 
is independent of the choice of $t_1, \dots, t_n$. 

In the rest of the paper we treat the boundary complex $\partial C_d(n)$ of $C_d(n)$ as an abstract simplicial complex. In particular, 
we identify a vertex $m(t_i)$ with $i \in [n] \coloneq \{1,2,\dots,n\}$ and 
the vertex set of $\partial C_d(n)$ with $[n]$. When listing the vertices in a face, we list them in increasing order.

The facets of $C_d(n)$ have a particularly nice description known as 
\emph{Gale's evenness condition}.

\begin{lemma}[Gale's Evenness Condition]
Let $F \subset [n]$ with $|F| = d$. Then $F$ is a facet of the cyclic polytope $C_d(n)$ if and only if for every pair $i < j$ with $i,j \notin F$, the number of elements of $F$ strictly between $i$ and $j$ is even.
\end{lemma}

\subsection{Squeezed balls and spheres}
Before we define squeezed balls and spheres, we need to establish some terminology. Suppose $F$ is a subset of $[n]$. We are first going to define the \emph{left}, \emph{right}, and \emph{middle sets} of $F$. Assume first that $F \neq [n]$, and let $j = \min\{i \in [n]: i \not \in F\}$ and $k = \max\{i \in [n]: i \not \in F\}$. Then $\{1, \ldots, j-1\}$ is the (possibly empty) \emph{left set} of $F$ and $\{k+1, \ldots, n\}$ is the (possibly empty) \emph{right set} of $F$. On the other hand, if $p$ and $q$ are elements in $[n] - F$, with $p < q$, and the interval $[p+1, q-1]$ is in $F$, then $[p+1, q-1]$ is a \emph{middle set} of $F$. In the special case $F = [n]$, we say that both the left set and the right set equal $F$, and that $F$ has no middle sets.

Let $d = 2t-1$ be odd. Let $\mathcal{F}^{[m,n]}_{2t}$ denote all the facets of the cyclic polytope $C_{2t}(n)$ that have even-sized left, middle and right sets, and whose vertices lie in $[m, n]$.  When $t$ and $n$ are fixed or understood from context, we abbreviate $\mathcal{F}^{[1,n]}_{2t}$ as $\mathcal{F}_{2t}$ or as $\mathcal{F}$. Since the elements of $\mathcal{F}$ are of the form $F = (a_1, a_1+1, \ldots , a_t, a_t+1)$, we can associate them to points of the form $(a_1, a_2, \ldots, a_t) \in \mathbb R^t$ such that $1 \leq a_1$, $a_1 + 1 < a_2$, $a_2+1 < a_3$, $\cdots$, $a_{t-1}+1 < a_t$ and $a_t +1 \leq n$. Plotting these integer points in $\mathbb R^t$ results in diagrams like the ones depicted in Figure \ref{fig:squeezed_ball_and_relative_squeezed_ball}.

The elements of $\mathcal{F}$ are ordered by the standard partial order $\leq_p$: for $F = (a_1, a_1+1, \ldots , a_t, a_t+1)$ and $G = (b_1, b_1+1, \ldots , b_t,b_t+1)$, we say that $G \leq_p F$ if $b_\ell \leq a_\ell$ for all $1 \leq \ell \leq t$. For an antichain $I$ in $\mathcal{F}$, we use $\mathcal{F}(I)$ to denote the initial set generated by $I$ \footnote{Given a partially ordered set $(\mathcal{F}, \le_p)$, the initial set generated by $I$ is the downward closure of $I$, i.e., $\mathcal{F}(I) = \{x \in \mathcal{F} : x \le y \text{ for some } y \in I\}.$}.

Now we can define squeezed balls and squeezed spheres following \cite{kalaiManyTriangulatedSpheres1988}: Let $I$ be an antichain in $\mathcal{F}$. The simplicial complex $B(I)$ whose facets are the sets in $\mathcal{F}(I)$ is a ball and it is called a \emph{squeezed ball}. The boundary complex $S(I) \coloneq \partial B(I)$ is called a \emph{squeezed sphere}. The facets of a squeezed ball can be plotted in $\mathbb R^t$ as in Figure \ref{fig:squeezed_ball_and_relative_squeezed_ball}a.

Squeezed balls form a large family of triangulated balls with remarkable combinatorial properties. In particular, we have the following theorems:

\begin{theorem}[\cite{kalaiManyTriangulatedSpheres1988}]
Every squeezed ball $B(I)$ is shellable.
\end{theorem}

\begin{theorem}[\cite{leeKalaisSqueezedSpheres2000}]
Every squeezed sphere $S(I)$ is shellable.
\end{theorem}

The following notation will be helpful when working with faces (most commonly facets) of these complexes:

\begin{notation}[\cite{leeKalaisSqueezedSpheres2000}]
    Let $\mathbb N$ denote the set of non-negative integers. Consider some element $j \in F \subset [n]$. Define 
    $$\ell(F, j) \coloneq \max\{i \in \mathbb N : i < j, i \not\in F\},$$ and $L(F, j) \coloneq (F-j)\cup \ell(F, j).$ Thus $L(F, j)$ is the set resulting from `pushing' $j$ and its immediate predecessors one step to the left. Similarly, define 
    $$r(F, j) \coloneq \min\{i \in \mathbb N : i > j, i \not\in F\},$$ and $R(F, j)\coloneq (F-j)\cup r(F,j).$ Thus $R(F, j)$ is the set resulting from `pushing' $j$ and its immediate successors one step to the right.
\end{notation}

\subsection{Relative squeezed balls and spheres}

Let $d = 2t-1$, and let $I$ be an antichain in $\mathcal{F}$. We define $I - 1$ as the antichain $\{x - \bm{1}_{2t} : x \in I, x_1 > 1\},$
where $\bm{1}_{2t}$ denotes the all-ones vector of length $2t$.

We say that two antichains $I$ and $J$ in $\mathcal{F}$ satisfy $J \leq_p I$ if $\mathcal{F}(J) \subseteq \mathcal{F}(I)$. Similarly, we say that 
$$J \prec_p I \text{ if } \mathcal{F}(J) \subseteq \mathcal{F}(I-1).$$

Now we can define relative squeezed balls and relative squeezed spheres following \cite{novikManyNeighborlySpheres2024}: Let $J \prec_p I$. The simplicial complex $B(I, J) \coloneq B(I)\setminus B(J)$, generated by the facets of $B(I)$ that do not belong to $B(J)$, is called a \emph{relative squeezed ball} and its facets can be plotted in $\mathbb R^t$ as in Figure \ref{fig:squeezed_ball_and_relative_squeezed_ball}b. Its boundary complex $S(I, J) \coloneq \partial\bigl(B(I) \setminus B(J)\bigr)$ is called a \emph{relative squeezed sphere}. The following result is relevant:

\begin{theorem}[\cite{novikManyNeighborlySpheres2024}] \label{thm:novik-zheng}
The simplicial complex $B(I, J)$ is a ball, and it is constructible. 
\end{theorem}

\begin{figure}[h!]
    \centering
    \begin{subfigure}{0.35\textwidth}
        \centering
        \includegraphics[width=\linewidth]{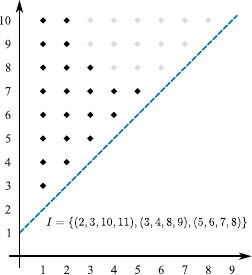}
        \caption{Diagram of a squeezed ball.}
    \end{subfigure}
    \hspace{2cm}
    \begin{subfigure}{0.35\textwidth}
        \centering
        \includegraphics[width=\linewidth]{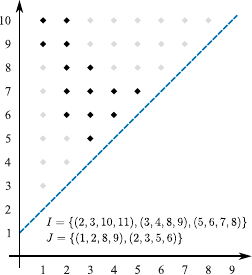}
        \caption{Diagram of a relative squeezed ball.}
    \end{subfigure}
    \caption{Diagrams of facets in $\mathcal F^{[1,11]}_4$, plotted as integer lattice points above the line $x+1 < y$.} 
    \label{fig:squeezed_ball_and_relative_squeezed_ball}
\end{figure}

\section{Shellability of relative squeezed spheres}\label{sec:spheres}

We begin this section by establishing two well-known shellability results that will be key elements in our proof. We then provide a characterization of the facets of a squeezed sphere $S(I)$, and use it to state a generalized version of Lee's shelling order for $S(I)$ and to prove that this generalized order is still a shelling. Next, we define a total order $\leq_J$ on the facets of $B(I)$ extending $\leq_p$, which will be used in the construction of a shelling order for $S(I,J)$. Finally, we provide a characterization of the facets of a relative squeezed sphere $S(I,J)$, establish a connection between the facets of $S(I, J)$ and the facets of $S(I)$ and $S(J)$, and use this, together with well-known shellability results, to prove that $S(I,J)$ is shellable.

We start by recalling two standard shellability results whose proofs can be found in the Appendix.

\begin{numlemma}{\ref{lemma:prefix-replacement}}{\rm(Prefix Replacement)}
    Let $\Delta$ be a shellable simplicial sphere or a shellable simplicial ball with a shelling order $F_1, \ldots, F_s$ and let $\Delta_1$ be the simplicial ball generated by an initial segment $F_1, \ldots, F_m$ of this shelling, where $1 \leq m \leq s$. Let $\Delta_2$ be a shellable simplicial ball with shelling $G_1, \ldots, G_r$. Suppose that $\partial \Delta_2 = \partial \Delta_1$ and that $\Delta' \coloneq \Delta_2 \cup (\Delta \setminus \Delta_1)$  is a simplicial ball or sphere in which no facet of $\Delta_2$ is a facet of $\Delta \setminus \Delta_1$. Then $\Delta'$ is shellable with shelling order $G_1, \ldots, G_r, F_{m+1}, \ldots, F_s$.
\end{numlemma}

\begin{numlemma}{\ref{lemma:reversibility}}{\rm(Reversibility)}
    If $F_1, F_2, \cdots, F_s$ is a shelling order for a simplicial sphere $\Delta$, then so is the reverse order $F_s, F_{s-1}, \cdots, F_1$.
\end{numlemma}

Now using the definition of the boundary complex, we can characterize the facets of a squeezed sphere as follows:

\begin{propositionDefinition}\label{prop:facets_squeezed_sphere}
    Suppose $d$ is odd and $B(I)$ is a squeezed $d$-ball. Let $F = (a_1, \ldots, a_{d+1})$ be a facet of $B(I)$. Then $F - a_i$ is a facet of $S(I)$ if and only if one of the following cases holds.
    \begin{enumerate}
        \item The index $i$ is even and $L(F, a_i) \not \in B(I)$. In this case $a_i$ must be in the left set of $F$. The set of all facets of $S(I)$ that satisfy this condition will be called $U_1(I)$. 
        \item The index $i$ is odd and $R(F, a_i) \not \in B(I)$. The set of all facets of $S(I)$ that satisfy this condition will be called $U_2(I)$.
    \end{enumerate} 
\end{propositionDefinition}

In Figure \ref{fig:squeezed_ball_diagram_boundary} we illustrate the elements of $S(I)$ as follows:
Let $F$ and $G$ be two facets in $\mathcal{F}$ that appear as adjacent dots in the diagram. Consider the $t$-dimensional unit cubes that enclose the dots associated to $F$ and $G$. The common face shared by these cubes in the diagram will be used to denote the intersection $F \cap G$. Note that since $F$ and $G$ are adjacent dots in the diagram, it follows that $|F \cap G| = 2t-1$. Therefore these new elements of the diagram will always correspond to ridges. In Figure \ref{fig:squeezed_ball_diagram_boundary}, we plot the ridges that form the boundary of $B(I)$ and use different colors to differentiate the elements that belong to $U_1(I)$ from the elements that belong to $U_2(I)$. 

\begin{figure}[h!]
    \centering
    \includegraphics[width=0.35\linewidth]{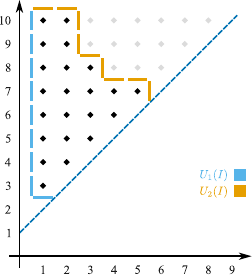}
    \caption{Diagram of $B(I)$ and its boundary $S(I) = U_1(I) \cup U_2(I)$.}
    \label{fig:squeezed_ball_diagram_boundary}
\end{figure}

Consider the following ordering for the facets of a squeezed sphere:

\begin{definition} \label{def:generalized_order}
Suppose $d$ is odd and $B(I)$ is a squeezed $d$-ball. Let $F = (a_1, \ldots, a_{d+1})$ and $G = (b_1, \ldots, b_{d+1})$ be facets of $B(I)$, let $<_{tot}$ denote an arbitrary total order that extends the standard partial order $<_p$ on $\mathcal F$, let $<_{RL}$ be the revlex order on $\mathcal F$, and let $F - a_k$ and $G - b_\ell$ be facets of $S(I)$. Set $r = \min\{j: a_i=b_i \text{ for all } i \geq j\}$. We say that $(G - b_\ell) <_{tot,\ RL} (F - a_k)$ if one of the following conditions is satisfied:
    \begin{enumerate}
        \item $k$ is odd and $\ell$ is even.
        \item $k$ and $\ell$ are both even and $G <_{tot} F$.
        \item $k$ and $\ell$ are both even, $G = F$, and $\ell > k$.
        \item $k$ and $\ell$ are both odd, $k = \ell \geq r$, and $G <_{RL} F$.
        \item $k$ and $\ell$ are both odd and $\ell > k \geq r$.
        \item $k$ and $\ell$ are both odd and $\ell \geq r > k$.
        \item $k$ and $\ell$ are both odd, $k,\ell < r$, and $F <_{RL} G$. 
    \end{enumerate}
\end{definition}

\begin{remark}\label{rmk:u1_before_u2}
    It follows from Condition 1 of Definition \ref{def:generalized_order} that, in the ordering of the facets of $S(I)$ given by $<_{tot,\ RL}$, the facets that belong to $U_1(I)$ come before the facets that belong to $U_2(I)$.
\end{remark}

Lee's main result in \cite{leeKalaisSqueezedSpheres2000} can then be stated as follows:

\begin{theorem}[\cite{leeKalaisSqueezedSpheres2000}]
    If $d$ is odd and $S(I)$ is a squeezed $(d-1)$-sphere, then $<_{RL,\ RL}$ is a shelling order of its facets.
\end{theorem}

The following stronger statement holds:

\begin{theorem} \label{thm:generalized_shelling}
    If $d$ is odd and $S(I)$ is a squeezed $(d-1)$-sphere, then $<_{tot,\ RL}$ is a shelling order of its facets for any total ordering $<_{tot}$ extending the partial order $<_p$ on $\mathcal F$.
\end{theorem}

\begin{proof}

    Notice that with this generalized ordering $<_{tot, \ RL}$, the facets $U_2(I)$ remain at the end of the shelling and in the same order as with $<_{RL,\ RL}$. Therefore, by the Prefix Replacement Lemma (see Lemma \ref{lemma:prefix-replacement}), it is enough to show that the new ordering of the facets in the prefix $U_1(I)$ is still a shelling. 

    The remainder of the proof closely follows the first half of the proof in \cite[Theorem 1]{leeKalaisSqueezedSpheres2000}:

    We will show that for each facet $F - a_k$ of $U_1(I)$ there is a subset $X$ that satisfies the two conditions from Definition \ref{def:restriction-face}. That is, 
    \begin{enumerate}
    \item for each $a_i \in X$ there is an earlier facet of $S(I)$ containing $(F - a_k) - a_i$.
    \item $X$ itself is not contained in any earlier facet.
    \end{enumerate}
    Let $F \in B(I)$ and $F-a_k$ be a facet of $U_1(I)$.
    Note that by the definition of $U_1(I)$, $k$ is even. Let $X = \{a_i : i > k$ and $i$ is even$\}$. 
    We first prove that $X$ satisfies the first condition. Take any $a_i \in X$. If $a_i$ is in the left set of $F$, then $F - a_i$ is an earlier facet of $S(I)$ (by condition 3) that contains $(F - a_k) - a_i$. On the other hand, if $a_i$ is not in the left set of $F$, then $L(F, a_i) \in B(I)$ and $L(F,a_i)-a_k$ is an earlier facet of $S(I)$ (by condition 2) that contains $(F-a_k)-a_i$. Hence $X$ satisfies the first condition.
    Now assume that $G - b_\ell$ is an earlier facet of $S(I)$ that contains $X$. Since $k$ is even, it follows that $(G - b_\ell) <_{tot,\ RL} (F - a_k)$ via condition 2 or 3.

    Condition 2 is impossible. Indeed, $X \subseteq G$ forces $F \leq_p G$, which contradicts $G <_{tot} F$. Condition 3 is also impossible, because $F = G$ and $\ell > k$ implies $b_\ell \in X$. Therefore, $X$ also satisfies the second condition.
\end{proof}

\begin{definition} \label{def:J_order}
    Let $J \prec_p I$ and let $F$ and $G$ be facets of $B(I)$. Let $<_{RL}$ denote the revlex order. We say that $G <_{J} F$ if one of the following conditions holds:
    \begin{enumerate}
        \item $G \in B(J)$ and $F \not \in B(J)$.
        \item $G$ and $F$ are both in $B(J)$ and $G <_{RL} F$.
        \item $G$ and $F$ are both in $B(I) \setminus B(J)$ and $G <_{RL} F$.
    \end{enumerate}
\end{definition}

\begin{proposition} \label{prop:J_order_extends_partial}
    The order $\leq_J$ on the facets of $B(I)$ is a total order that extends the partial order $\leq_p$.
\end{proposition}
\begin{proof}
    By conditions 1--3, $<_J$ places every facet of $B(J)$ before every facet of $B(I) \setminus B(J)$, and orders each of these two blocks by the revlex order $<_{RL}$. Since the two blocks partition the facets of $B(I)$ and $<_{RL}$ is a strict total order on each of them, $<_J$ is irreflexive, transitive and total; that is, $\leq_J$ is a total order. To see that it extends $<_p$, assume $G <_p F$. It is enough to show that $G <_J F$.
    
    If $F \in B(J)$, then the assumption $G<_p F$ implies that $G$ is also in $B(J)$. Since the revlex order extends the partial order and $G <_p F$, then $G <_{RL} F$. In this case, $G <_J F$ by condition 2.
    
    Assume $F \not \in B(J)$. If $G \in B(J)$, it follows that $G <_J F$ by condition 1. If $G \not \in B(J)$, then since the revlex order extends the partial order, we have $G <_J F$ by condition 3. 
\end{proof}
\begin{remark}
    Note that the revlex order was chosen arbitrarily in Definition \ref{def:J_order} and replacing it with any other total order extending the partial order would result in an order that still satisfies Proposition \ref{prop:J_order_extends_partial}.
\end{remark}

\begin{proposition} \label{prop:J_shelling}
    Let $J \prec_p I$ and let $U_1(I)$ and $U_2(I)$ be the sets defined in Proposition \ref{prop:facets_squeezed_sphere}. We define
    $$U_{11}(I,J) \coloneq \{F - a_i \in U_1(I) : F \in B(J)\},$$
    $$U_{12}(I,J) \coloneq \{F - a_i \in U_1(I) : F \not \in B(J)\}.$$
    There exists a shelling order of $S(I)$ such that the facets in $U_{11}(I,J)$ come first, the facets in $U_{12}(I,J)$ come second, and the facets in $U_2(I)$ come third.
\end{proposition}
\begin{proof}
    Let $<_J$ be the total order defined in Definition \ref{def:J_order}. Then $<_{J,\ RL}$ is a shelling order of $\partial B(I)$ by Theorem \ref{thm:generalized_shelling}. 

    Let $F-a_k$ be an arbitrary facet in $U_{11}(I,J)$, let $G-b_\ell$ be an arbitrary facet in $U_{12}(I,J)$, and let $H-c_m$ be an arbitrary facet in $U_2(I)$. It follows that 
    \begin{itemize}
        \item $(F-a_k) <_{J,\ RL} (G-b_\ell)$ by condition 2 of Definition \ref{def:generalized_order}. Therefore, the facets in $U_{11}(I,J)$ come before the facets in $U_{12}(I,J)$.
        \item $(G-b_\ell) <_{J,\ RL} (H-c_m)$ by condition 1 of Definition \ref{def:generalized_order}. Therefore, the facets in $U_{12}(I,J)$ come before the facets in $U_2(I)$.
    \end{itemize}
\end{proof}

Now, using the definition of the boundary complex, we can characterize the facets of a relative squeezed sphere as follows:

\begin{propositionDefinition}\label{prop:facets_relative_squeezed_spheres}
Suppose $B(I,J)$ is a relative squeezed $d$-ball, with $d$ odd and $J \prec_p I$. Let $F$ be a facet of $B(I,J)$. Then $F - a_i$ is a facet of $\partial B(I,J)$ if and only if one of the following cases holds:
\begin{enumerate}
    \item The index $i$ is even and $L(F, a_i) \not \in B(I) \setminus B(J)$. In this scenario there are two cases:
    \begin{enumerate}
        \item $L(F, a_i) \in B(J)$. The set of all facets of $\partial B(I,J)$ that satisfy this condition will be called $W_1(I,J)$.
        \item $L(F, a_i) \not \in B(J)$, and so $L(F,a_i) \not \in B(I)$. In this case $a_i$ must be in the left set of $F$. The set of all facets of $\partial B(I,J)$ that satisfy this condition will be called $W_2(I,J)$.  
    \end{enumerate}
    \item The index $i$ is odd and $R(F, a_i) \not \in B(I) \setminus B(J)$. In this scenario we must have $R(F, a_i) \not \in B(I)$. The set of all facets of $\partial B(I,J)$ that satisfy this condition will be called $W_3(I,J)$.
\end{enumerate} 
\end{propositionDefinition}

Recall that the notation $U_1, U_2$ was defined in Proposition \ref{prop:facets_squeezed_sphere}, the notation $U_{11}, U_{12}$ was defined in Proposition \ref{prop:J_shelling}, and the notation $W_1, W_2, W_3$ was defined in Proposition \ref{prop:facets_relative_squeezed_spheres}. The boundary of a relative squeezed ball, as well as all these sets, is illustrated in Figure \ref{fig:relative_squeezed_ball_diagram_boundary} below. We also illustrate a few equalities among these sets that will be formally proven in Proposition \ref{prop:matching_sets}.

\begin{figure}[h!]
    \centering
    \includegraphics[width=0.35\linewidth]{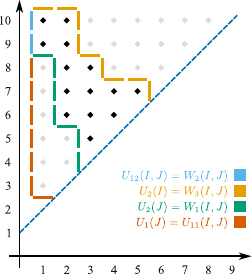}
    \caption{Diagram of $B(I,J)$ and its boundary $S(I, J) = W_1(I,J) \cup W_2(I,J) \cup W_3(I,J)$.}
    \label{fig:relative_squeezed_ball_diagram_boundary}
\end{figure}

The following observation establishes a connection between the facets of $S(I, J)$ and the facets of $S(I)$ and $S(J)$.
\begin{proposition} \label{prop:matching_sets} Suppose $B(I,J)$ is a relative squeezed $d$-ball, with $d$ odd and $J \prec_p I$. The following equalities hold:
    \begin{enumerate}
        \item $W_1(I, J) = U_2(J)$.
        \item $W_2(I, J) = U_{12}(I,J)$.
        \item $W_3(I,J) = U_2(I)$.
        \item $U_{11}(I,J) = U_1(J).$
    \end{enumerate}
\end{proposition}
\begin{proof}
    For (1), let $F - a_i \in W_1(I,J)$. Let $F' = L(F, a_i)$ and $a'_j = \ell(F, a_i)$. Note that $F' - a'_j = F - a_i$. Furthermore, $F' \in B(J)$, $j$ is odd, $R(F', a'_j) \not \in B(J)$.  This means that $F - a_i=F' - a'_j \in U_2(J)$, and so $W_1(I,J) \subseteq U_2(J)$.

    Now let $F' - a'_j \in U_2(J)$. Let $F = R(F', a'_j)$ and $a_i = r(F', a'_j)$. Note that $F \not \in B(J)$, $i$ is even, and $L(F,a_i) \in B(J)$. Furthermore, since $F' \in B(J) \subseteq B(I-1)$, it follows that $F \in B(I)$. This means that $F'-a'_j = F-a_i \in W_1(I,J)$, and so $U_2(J) \subseteq W_1(I,J)$.

    For the set equalities (2), (3) and (4), the sets are equal by definition:

    For (2), let $F - a_i \in W_2(I,J)$. Then $F \in B(I) \setminus B(J)$, $i$ is even, and $L(F, a_i) \not \in B(I)$. Since $F \in B(I)$, $i$ is even, and $L(F, a_i) \not \in B(I)$, it follows that $F-a_i \in U_1(I)$. Since $F \not \in B(J)$, we conclude that $F - a_i \in U_{12}(I,J)$. The other inclusion also follows since all the implications are reversible.

    For (3), let $F - a_i \in W_3(I,J)$. Then $F \in B(I)\setminus B(J)$, $i$ is odd, and $R(F, a_i) \not \in B(I)$. Since $F \in B(I)$, $i$ is odd, and $R(F, a_i) \not \in B(I)$, it follows that $F - a_i \in U_2(I)$. 
    
    Now let $F - a_i \in U_2(I)$. Then $F \in B(I)$, $i$ is odd, and $R(F, a_i) \not \in B(I)$. Note that $F$ cannot be in $B(J) \subseteq B(I-1)$, because in that case $R(F, a_i)$ would be in $B(I)$. Since $F \in B(I) \setminus B(J)$, $i$ is odd, and $R(F, a_i) \not \in B(I)$, we conclude that $F - a_i \in W_3(I, J)$.

    For (4), let $F - a_i \in U_{11}(I,J)$. Then $F - a_i \in U_1(I)$ and $F \in B(J)$. The condition $F-a_i\in U_1(I)$ implies that $i$ is even and $a_i$ is in the left set of $F$. Since $F \in B(J)$, $i$ is even, and $a_i$ is in the left set of $F$, we obtain that $F-a_i \in U_1(J)$. The other inclusion also follows since all the implications are reversible.
\end{proof}

With Proposition \ref{prop:matching_sets} and Lemmas \ref{lemma:prefix-replacement} and \ref{lemma:reversibility} in hand, we are ready to prove the main result of this section. 
\begin{theorem} \label{thm:relative_squeezed_sphere_shellable}
    Suppose $d$ is odd and $J \prec_p I$. Let $B(I,J)$ be a relative squeezed $d$-ball. Then, the relative squeezed sphere $\partial B(I,J)$ is shellable.
\end{theorem}

\begin{proof}
    By Remark \ref{rmk:u1_before_u2}, there exists a shelling order of $\partial B(J)$ such that all the facets in $U_1(J)$ come before all the facets in $U_2(J)$. Since this is a shelling of a sphere, reversing this order results in another shelling (see Lemma \ref{lemma:reversibility}) in which $U_2(J)$ is an initial segment, and so $\overline{U_2(J)}$ is shellable. Furthermore, since $\overline{U_1(J) \cup U_2(J)}$ is a sphere, $\overline{U_1(J)}$ and $\overline{U_2(J)}$ must be balls with the same boundary. 

    By Proposition \ref{prop:J_shelling}, there exists a shelling order of $\partial B(I)$ such that the facets in $U_{11}(I,J)$ come first, the facets in $U_{12}(I,J)$ come second, and the facets in $U_2(I)$ come third. 
    
    Since earlier we established that $\partial \overline{U_2(J)} = \partial \overline{U_1(J)}$, and since $\partial\overline{U_1(J)} = \partial \overline{U_{11}(I,J)}$ by Proposition \ref{prop:matching_sets}, we obtain that
    $$\partial \overline{U_2(J)} = \partial \overline{U_{11}(I,J)}.$$
    Moreover, by Proposition \ref{prop:matching_sets}, the complex $(\partial B(I) \setminus \overline{U_{11}(I,J)}) \cup \overline{U_2(J)}$ is precisely
    \begin{align*}
        (\partial B(I) \setminus \overline{U_{11}(I,J)}) \cup \overline{U_2(J)} &= \overline{U_{2}(J) \cup U_{12}(I,J)\cup U_2(I)}\\
        &= \overline{W_1(I,J) \cup W_2(I,J) \cup W_3(I,J)}\\
        &= \partial B(I,J),
    \end{align*} 
    which is a simplicial sphere, since $B(I,J)$ is a ball by Theorem \ref{thm:novik-zheng}.\footnote{The closing remark of Section \ref{sec:balls} provides an alternative proof, independent of this section, that $B(I,J)$ is a ball.} Therefore, Lemma \ref{lemma:prefix-replacement}, applied with $\Delta = \partial B(I)$, $\Delta_1 = \overline{U_{11}(I,J)}$ and $\Delta_2 = \overline{U_2(J)}$, implies that this complex is shellable.

    We conclude that  $\partial B(I, J)$ is a shellable sphere, as desired.
\end{proof}

\section{Shellability of relative squeezed balls}\label{sec:balls}

We begin this section by introducing some notation:

\begin{notation}
    Let $F\coloneq(a_1, \ldots, a_{2t})$ be an element of $\mathcal F^{[1,n]}_{2t}$ and let $i$ and $j$ be integers. We denote by $F_{[i,j]}$ the subset 
    $$F_{[i,j]} \coloneq (a_i, a_{i+1} \ldots, a_{j-1}, a_{j}).$$
    If $i > j,$ then we say that $F_{[i,j]} = \emptyset$.
\end{notation}

\begin{notation}[\cite{novikManyNeighborlySpheres2024}]
Let $I$ be an antichain of $\mathcal F ^{[1,n]}_{2t}$. Let $M \subseteq [n]$ be a subset of size $2m$, and let
$$S_M(I) \coloneq \{N \subseteq [n] : \max(M) < \min(N) \text{ and } M \cup N \in \mathcal F(I)\}.$$
We define the antichain $I(M)$ of $\mathcal F ^{[1,n]}_{2(t-m)}$ as follows:
$$I(M) \coloneq \text{maximal elements of } S_M(I) \text{ with respect to } <_p.$$
\end{notation}

\begin{notation}[\cite{novikManyNeighborlySpheres2024}]
If $I$ is an antichain in $\mathcal F$ and $k \in \mathbb N$, we denote by $B(I, k)$ the simplicial complex generated by all the facets $F$ in $\mathcal F(I)$ such that $\min(F) \geq k$.
\end{notation}

As an example, let $I$ be an antichain and let $F\coloneq (a_1, \ldots, a_{2t})$ be a facet of $B(I)$. Consider the simplicial complex $B(I(F_{[1,2j]}), a_{2j}+1)$. According to the definitions above, the facets of this complex are of the form $G_{[2j+1, 2t]}$, where $G$ is a facet of $B(I)$ that has $F_{[1,2j]}$ as a prefix (in other words, $G_{[1,2j]} = F_{[1,2j]}$). In particular, we can see that
$$F \in B(I) \iff F_{[2j+1, 2t]} \in B(I(F_{[1,2j]}), \, a_{2j}+1).$$

As we will state in Proposition \ref{prop:iso-squeezed-balls}, simplicial complexes of this form are isomorphic to squeezed balls, and they will appear frequently throughout the remainder of the paper.

\begin{notation}
    Let $k \in \mathbb Z$. If $x \in \mathbb Z$, we define $\varphi_k(x) \coloneq x - k$. Similarly, if $F$ is a set with $\min(F) \geq k+1$, we denote by $\varphi_k(F)$ the set given by $\{\varphi_k(x) : x \in F\}$, and if $\mathcal F$ is a collection of sets, we denote by $\varphi_k(\mathcal F)$ the collection of sets given by $\{\varphi_k(F): F \in \mathcal F \text{ and } \min(F) \geq k+1\}$.
\end{notation}

\begin{proposition}\label{prop:iso-squeezed-balls}
    The complex $B(I, k+1)$ is isomorphic to the squeezed ball $B(\varphi_k(I))$ via the map $\varphi_k$.
\end{proposition}

\begin{corollary} \label{cor:iso-relative-squeezed-balls}
    The complex $B(I, k+1) \setminus B(J, k+1)$ is isomorphic to the relative squeezed ball $B(\varphi_k(I)) \setminus B(\varphi_k(J))$ via the map $\varphi_k$.
\end{corollary}

With these results and notations in hand, we are now ready to introduce the ordering that will serve as a shelling order for a relative squeezed ball.

\begin{definition}\label{def:B-ordering}
    Let $B \coloneq B(I) \setminus B(J)$ be a relative squeezed ball of dimension $d \coloneq 2t-1$. Let $<_{tot}$ be a total order that extends the partial order $<_p$, and let $F$ and $G$ be facets of $B$. We define $<_B$ recursively according to the following rules:

    \begin{enumerate}
        \item If $G \not \in B(I-1)$ and $F \in B(I-1)$, then $G <_B F$.
        \item If $F, G \in B(I-1)$, then $G <_B F$ if and only if $F <_{tot} G$.
        \item If $F, G \not \in B(I-1)$ and $\min(G) < \min(F)$, then $G <_B F$.
        \item If $F, G \not \in B(I-1)$ and $\min(G) = \min(F) = k$, then let 
        \begin{align*}
            F' &\coloneq \varphi_{k+1}(F \setminus \{k, k+1\}), \\
            G' &\coloneq \varphi_{k+1}(G \setminus \{k, k+1\}), \\
            B' &\coloneq \varphi_{k+1}(B(I(\{k, k+1\}), k+2)\setminus B((I-1)(\{k, k+1\}), k+2)).
        \end{align*}
        Note that $F', G' \in B'$. Furthermore, by Corollary \ref{cor:iso-relative-squeezed-balls}, $B'$ is a relative squeezed ball of dimension $2(t-1)-1$. In this case, we define $<_B$ recursively by: 
        $$G <_B F \text{ if } G' <_{B'} F'.$$
    \end{enumerate}

\end{definition}

\begin{remark}
For $t = 1$, this definition reduces to the reverse order: $G <_B F$ if and only if $\min(G) > \min(F)$. In particular, when $t = 1$, the definition of $<_B$ is non-recursive, since rule~4 never applies. Thus, the case $t = 1$ serves as the implicit base case, ensuring that the definition is well-founded.
\end{remark}

In plain English, Definition \ref{def:B-ordering} above lists the facets of a relative squeezed ball in the following order: First list the facets that belong to $B(I) \setminus B(I-1)$. The way this is done is by first listing all the facets of $B(I)\setminus B(I-1)$ whose minimum element is 1, then listing all the facets of $B(I)\setminus B(I-1)$ whose minimum element is 2, and so on. The set of all facets of $B(I)\setminus B(I-1)$ whose minimum element is $i$ is of the form $\{i, i+1\} * \varphi_{i+1}^{-1}(B')$, where $B'$ is isomorphic to a lower-dimensional relative squeezed ball by Corollary \ref{cor:iso-relative-squeezed-balls}. Therefore, the facets in each of these sets can be listed recursively. After listing all the facets of  $B(I)\setminus B(I-1)$ in this manner, list the remaining facets of $B$ from \textbf{larger to smaller} following any fixed total ordering $<_{tot}$ that extends the partial order $<_p$. 

Figure \ref{fig:relative_squeezed_ball_diagram_top_layer} below demonstrates this ordering for a $3$-dimensional relative squeezed ball, where we let $<_{tot}$ be the lexicographic order. 

\begin{figure}[h!]
    \centering
    \includegraphics[width=0.35\linewidth]{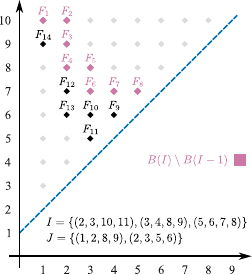}
    \caption{Ordering of $B(I)\setminus B(J)$ with $B(I)\setminus B(I-1)$ highlighted.}
    \label{fig:relative_squeezed_ball_diagram_top_layer}
\end{figure}

Iterating the recursion of Condition 4 of Definition \ref{def:B-ordering} will be convenient in what follows, so we fix notation for it.

\begin{notation}\label{not:tower}
    Let $B \coloneq B(I)\setminus B(J)$ be a relative squeezed ball of dimension $d \coloneq 2t-1$, and let $F \coloneq (a_1, \ldots, a_{2t})$ and $G$ be facets of $B$. Set
    $$F^0 \coloneq F, \qquad G^0 \coloneq G, \qquad B^0 \coloneq B,$$
    and, for as long as Condition 4 of Definition \ref{def:B-ordering} applies to $F^j$ and $G^j$ as facets of $B^j$,
    $$F^{j+1} \coloneq (F^j)', \qquad G^{j+1} \coloneq (G^j)', \qquad B^{j+1} \coloneq (B^j)',$$
    where the \emph{prime} operation is the one defined in Condition 4. We write $B^j = B(I^j) \setminus B(J^j)$.

    Since Condition 4 requires $\min(F^j) = \min(G^j)$, this sequence is defined up to level $j$ only if $F_{[1,2j]} = G_{[1,2j]}$; in that case, unwinding the definition of $\varphi$ gives
    $$F^j = \varphi_{a_{2j}}(F_{[2j+1,2t]}), \qquad G^j = \varphi_{a_{2j}}(G_{[2j+1,2t]}), \qquad I^j = \varphi_{a_{2j}}\bigl(I(F_{[1,2j]})\bigr),$$
    and therefore $B(I^j) = \varphi_{a_{2j}}\bigl(B(I(F_{[1,2j]}),\, a_{2j}+1)\bigr)$ by Proposition \ref{prop:iso-squeezed-balls}.
\end{notation}

We now state the main result of this section.

\begin{theorem} \label{thm:relative_squeezed_ball_shellable}
    Suppose $d$ is odd and $J \prec_p I$ be antichains. Let $B(I,J)\coloneq B(I) \setminus B(J)$. Then, $B(I,J)$ is a shellable ball and the total order $<_B$ from Definition \ref{def:B-ordering} is a shelling order of its facets.
\end{theorem}

To show that this ordering is a shelling order, we will use the Restriction Face Criterion described in Lemma \ref{lemma:restriction-face-criterion}. In fact, we will see that under this ordering, every facet $F$ has a restriction face: the face $X_B(F)$ defined below.

\begin{definition}\label{def:X_B(F)}
    Let $B \coloneq B(I) \setminus B(J)$ be a relative squeezed ball of dimension $d \coloneq 2t-1$. Let $F \coloneq (a_1, \ldots, a_{2t})$ be a facet of $B$. We define the set $X_B(F) \subseteq F$ as follows:
    \begin{enumerate}
        \item If $F \in B(I)\setminus B(I-1)$ and $\min(F) = 1$, let 
        $$X_B(F) = \varphi^{-1}_2(X_{B'}(F')),$$ 
        where $F' = \varphi_2(F \setminus \{1,2\})$ and $B' = \varphi_{2}(B(I(\{1,2\}), 3) \setminus B((I-1)(\{1,2\}), 3))$. As our base case in this recursion, for $t=1$, we let $X_B(F) = \emptyset$.
        \item Otherwise, let 
        $$X_B(F) = \{a_{2i} : 1 \leq i < m\} \cup \{a_{2i-1} : m < i \leq t\},$$ 
        where $m = \min\{i : (a_1, \ldots, a_{2i}) \cup (a_{2i+1}+1, \ldots, a_{2t}+1) \in B(I)\}$. 
    \end{enumerate}
\end{definition}

Intuitively, $X_B(F)$ is recursively defined in the first case because, according to Definition \ref{def:B-ordering}, the set 
$$\{F : F \in B(I)\setminus B(I-1) \text{ and } \min(F) = 1\}$$
is an initial segment with respect to the ordering $<_B$. Furthermore, this is a set of the form $\{1,2\} * B'$, where $B'$ is isomorphic to a lower-dimensional relative squeezed ball. Therefore, minimal new faces in this initial segment should be in correspondence with the minimal new faces in the ordering of the lower-dimensional ball $B'$. 

An alternative way to characterize the $m$ from the second case is through the following lemma, which we will use in the proof of Proposition \ref{prop:restriction-face-condition-1}. The proof of this lemma can be found in the Appendix.

\begin{numlemma}{\ref{lemma:about-F-and-m}}
    Let $F \coloneq (a_1, \ldots, a_{2t})$ be a facet of $B \coloneq B(I) \setminus B(J)$, and let $m = \min\{i : (a_1, \ldots, a_{2i})\cup (a_{2i+1}+1, \ldots, a_{2t}+1) \in B(I)\}$. Then, for $0\leq j < m$,
    $$F_{[2j+1, 2t]} \in B(I(F_{[1,2j]}), a_{2j}+1)\setminus B(I(F_{[1,2j]}) - 1, a_{2j}+1).$$
    Furthermore, for $j=m$,
    $$F_{[2j+1, 2t]} \in B(I(F_{[1,2j]}) - 1, a_{2j}+1).$$
\end{numlemma}

Let $F$ be a facet of a relative squeezed ball $B$. To show that $X_B(F)$ is the restriction face of $F$, we need to prove that $X_B(F)$ satisfies the two conditions outlined in Definition \ref{def:restriction-face}. We devote the rest of this section to proving these two conditions. 

\subsection{Verifying the restriction face criterion: Part 1}

To verify Condition 1 we need to show that for every $x \in X_B(F)$, there exists a facet $G$ of $B$ such that $G<_BF$ and $F\setminus\{x\} \subseteq G$. In Definition \ref{def:G_B(F,x)} below, we define a set $G_B(F,x)$ that will serve this purpose. Specifically, in Proposition \ref{prop:restriction-face-condition-1}, we prove that $G_B(F,x)$ is a facet of $B$ such that $G_B(F,x) <_B F$ and $F\setminus\{x\} \subseteq G_B(F,x)$. Along the way, we use Lemmas \ref{lemma:G_case_2a} and \ref{lemma:G_case_2b} as technical tools.

\begin{definition} \label{def:G_B(F,x)}
    Let $B \coloneq B(I) \setminus B(J)$ be a relative squeezed ball of dimension $d \coloneq 2t-1$. Let $F \coloneq (a_1, \ldots, a_{2t})$ be a facet of $B$, and let $x \in X_B(F)$. We define $G_B(F,x)$ as follows:
    \begin{enumerate}
        \item If $F \in B(I)\setminus B(I-1)$ and $\min(F) = 1$, then $$G_B(F,x) = \{1,2\}\cup \varphi^{-1}_2(G_{B'}(F',x')),$$ where $B' = \varphi_{2}(B(I(\{1,2\}), 3) \setminus B((I-1)(\{1,2\}), 3))$, $F' = \varphi_2(F \setminus \{1,2\})$, and $x' = \varphi_2(x) \in X_{B'}(F')$.
        \item Otherwise, we consider the following two cases:
        \begin{enumerate}
            \item If $x \in \{a_{2i} : 1 \leq i < m\}$, then $G_B(F,x) = L(F,x)$.
            \item If $x \in \{a_{2i-1} : m < i \leq t\}$, then $G_B(F,x) = R(F, x)$.
        \end{enumerate}
    \end{enumerate}
\end{definition}

The following two lemmas describe some properties of $G_B(F,x)$ when $F$ and $x$ fall under case 2a and 2b of Definition \ref{def:G_B(F,x)}. To avoid losing sight of the main goal of this section, the proofs of these lemmas can be found in the Appendix. 

\begin{numlemma}{\ref{lemma:G_case_2a}}
    Let $F \coloneq (a_1, \ldots, a_{2t})$ be a facet of $B \coloneq B(I) \setminus B(J)$, let $x \in X_B(F)$, let $G \coloneq G_B(F, x)$ be defined according to case 2a of Definition \ref{def:G_B(F,x)}, and let $r = \max\{j : F_{[1,2j]}=G_{[1,2j]}\}$. Then, for $0 \leq j \leq r$,
    $$G_{[2j+1, 2t]} \in B(I(F_{[1,2j]}), a_{2j}+1)\setminus B(I(F_{[1,2j]}) - 1, a_{2j}+1).$$

\end{numlemma}

\begin{numlemma}{\ref{lemma:G_case_2b}}
    Let $F \coloneq (a_1, \ldots, a_{2t})$ be a facet of $B \coloneq B(I) \setminus B(J)$, let $x \in X_B(F)$, let $G \coloneq G_B(F, x)$ be defined according to case 2b of Definition \ref{def:G_B(F,x)}, and let $r = \max\{j : F_{[1,2j]}=G_{[1,2j]}\}$. Recall that $m = \min\{i : (a_1, \ldots, a_{2i})\cup (a_{2i+1}+1, \ldots, a_{2t}+1) \in B(I)\}$. Then, for $0 \leq j \leq r$,
    $$G_{[2j+1, 2t]} \in B(I(F_{[1,2j]}), a_{2j}+1)\setminus B(J(F_{[1,2j]}), a_{2j}+1).$$
    Furthermore, for $0 \leq j < m$,
    $$G_{[2j+1, 2t]} \in B(I(F_{[1,2j]}), a_{2j}+1)\setminus B(I(F_{[1,2j]})-1, a_{2j}+1).$$
\end{numlemma}

Now we are ready to prove the main result of this section. The structure of the proof is as follows: 

If $F$ falls under case 1 of Definition \ref{def:G_B(F,x)}, then the proof is reduced to a lower-dimensional case where the statement holds by induction.

If $F$ and $x$ fall under cases 2a or 2b of Definition \ref{def:G_B(F,x)}, then the proof consists of two steps. Step 1 is showing that there is a scenario where the result follows directly without the need of recursion, and Step 2 is showing that otherwise the problem can be recursively reduced (using Lemmas \ref{lemma:G_case_2a} and \ref{lemma:G_case_2b}) to the scenario from Step 1.

\begin{proposition} \label{prop:restriction-face-condition-1}
    Let $G \coloneq G_B(F, x)$. Then $G$ is a facet of $B$, $F\setminus \{x\} \subseteq G$, and $G <_B F$.
\end{proposition}

\begin{proof}
    We argue by induction on $t$. According to Definition \ref{def:G_B(F,x)}, we must consider the following cases:

    \begin{enumerate}
        \item \textbf{Suppose that $\bm{F \in B(I) \setminus B(I-1)}$ and $\bm{\min(F) = 1}$.}
        
        In this case 
        $$G_B(F,x) = \{1,2\}\cup \varphi^{-1}_2(G_{B'}(F',x')).$$ 

        Since $B'$ is a relative squeezed ball of dimension $2(t-1)-1$, by the induction hypothesis, $G_{B'}(F',x')$ is a facet of $B'$ that contains $F'\setminus\{x'\}$ and comes before $F'$.
        
        Since $G_{B'}(F', x')$ is a facet of $B'$, the set $G_B(F,x)$ must be a facet of $B$. Furthermore, since $G_{B'}(F',x')$ contains $F' \setminus \{x'\}$, it follows that
        $$G_B(F,x) = \{1,2\} \cup \varphi^{-1}_2(G_{B'}(F',x'))  \supseteq \{1,2\} \cup \varphi^{-1}_2(F'\setminus \{x'\}) = \{1,2\} \cup ((F\setminus\{1,2\})\setminus\{x\}) = F \setminus \{x\}.$$
        Finally, since $G_{B'}(F', x') <_{B'} F'$ and $\min(F) = \min(G) = 1$, we can conclude that $G_B(F, x) <_B F$ by Condition 4 of Definition \ref{def:B-ordering}.
        \item \textbf{Otherwise, we consider the following two cases:}
        \begin{enumerate}
            \item \textbf{Suppose that $\bm{x \in \{a_{2i} : 1 \leq i < m\}}$.}

            In this case, $G$ contains $F\setminus\{x\}$ because $G = L(F,x)$. Furthermore, by taking $j=0$ in Lemma \ref{lemma:G_case_2a} we see that $G \in B(I)\setminus B(I-1)$, and so it is a facet of $B$. It remains to prove that $G <_B F$. For this, we consider the following two scenarios:

            \begin{enumerate}
                \item Suppose that $\min(G) < \min(F)$. By Lemma \ref{lemma:about-F-and-m}, $F \in B(I) \setminus B(I-1)$. Since $G$ is also in $B(I) \setminus B(I-1)$, Condition 3 of Definition \ref{def:B-ordering} guarantees that $G <_B F$.

                \item Otherwise, it must be the case that $\min(G) = \min(F) = k$. We will show that Condition 4 of Definition \ref{def:B-ordering} can be applied iteratively to reduce this scenario to scenario i. For this, let $r = \max\{j : F_{[1,2j]} = G_{[1,2j]}\}$, and let $F^j$, $G^j$ and $B^j = B(I^j)\setminus B(J^j)$ be as in Notation \ref{not:tower}. Lemmas \ref{lemma:about-F-and-m} and \ref{lemma:G_case_2a} imply that for $0 \leq j \leq r$, both $F^j$ and $G^j$ belong to $B(I^j)\setminus B(I^j-1)$, so Condition 4 of Definition \ref{def:B-ordering} applies at each level and yields
                \begin{align*}
                    G <_B F &\iff G^1 <_{B^1} F^1 \iff G^2 <_{B^2} F^2 \iff \cdots \iff G^r <_{B^r} F^r.
                \end{align*}
                Since $r = \max\{j : F_{[1,2j]} = G_{[1,2j]}\}$ and since $G = L(F, x)$, we must have $\min(G^r) < \min(F^r)$. Therefore, this case is reduced to scenario i.
            \end{enumerate}
            From this, we conclude that $G <_B F,$ as desired.

            \item \textbf{Suppose that $\bm{x \in \{a_{2i-1} : m < i \leq t\}}$.}
            
            In this case $G$ contains $F\setminus\{x\}$ because $G = R(F,x)$. Furthermore, by taking $j=0$ in Lemma \ref{lemma:G_case_2b} we see that $G \in B(I)\setminus B(J)$, and so it is a facet of $B$. It remains to prove that $G <_B F$. For this, we consider the following two scenarios:
            \begin{enumerate}
                \item Suppose that $m=0$. In this case, $F \in B(I-1)$. If $G \notin B(I-1)$, then $G <_B F$ by Condition 1 of Definition \ref{def:B-ordering}. If $G \in B(I-1)$, then, since $F <_p G$, it follows from Condition 2 of Definition \ref{def:B-ordering} that $G <_B F$.
                \item Suppose that $m>0$. Since $F_{[1,2m]} = G_{[1,2m]}$, we may define $F^j$, $G^j$ and $B^j = B(I^j)\setminus B(J^j)$ for $0 \leq j \leq m$ as in Notation \ref{not:tower}. Lemmas \ref{lemma:about-F-and-m} and \ref{lemma:G_case_2b} imply that for $0\leq j < m$, both $F^j$ and $G^j$ belong to $B(I^j) \setminus B(I^j-1)$, so Condition 4 of Definition \ref{def:B-ordering} applies at each level and yields
                \begin{align*}
                    G <_B F &\iff G^1 <_{B^1} F^1 \iff G^2 <_{B^2} F^2 \iff \cdots \iff G^m <_{B^m} F^m.
                \end{align*}
                Lemma \ref{lemma:about-F-and-m} guarantees that $F^m \in B(I^m -1)$, hence the same argument as in scenario i applies. Specifically, if $G^m \notin B(I^m-1)$, then $G^m <_{B^m} F^m$ by Condition 1 of Definition \ref{def:B-ordering}. On the other hand, if $G^m \in B(I^m-1)$, then, since $F^m <_p G^m$, it follows from Condition 2 of Definition \ref{def:B-ordering} that $G^m <_{B^m} F^m$.
            \end{enumerate}
            Thus, we conclude that $G <_B F,$ as desired.
        \end{enumerate}
    \end{enumerate}
\end{proof}

\subsection{Verifying the restriction face criterion: Part 2}

With Proposition \ref{prop:restriction-face-condition-1}, we have proven that $X_B(F)$ indeed satisfies the first condition for being the restriction face of $F$. In Proposition \ref{prop:restriction-face-condition-2} below, we show that $X_B(F)$ also satisfies the second condition for being the restriction face of $F$. But first, we state a technical lemma which will be proved in the Appendix:

\begin{numlemma}{\ref{lemma:condition-2}}
    Let $F\coloneq (a_1, \ldots, a_{2t})$ be a facet of $B$ such that $F \in B(I) \setminus B(I-1)$ and $\min(F) >1$, and let $G$ be a facet of $B$ containing $X_B(F)$. If $\min(G) < \min(F)$, then $G \in B(I-1)$.
\end{numlemma}

\begin{proposition} \label{prop:restriction-face-condition-2}
    Let $F$ be a facet of $B \coloneq B(I) \setminus B(J)$. Then $X_B(F)$ is not contained in any facet $G$ of $B$ such that $G <_B F$.
\end{proposition}

\begin{proof}
    We argue by induction on $t$. Consider the following cases according to Definition \ref{def:X_B(F)}:
    \begin{enumerate}
        \item \textbf{Suppose $\bm{F \in B(I)\setminus B(I-1)}$ and $\bm{\min(F)=1}$.}
        Let $G$ be a facet of $B$ such that $G <_B F$. It is enough to show that $G$ does not contain $X_B(F)$. 

        Observe that $G <_B F$ can occur only via Condition 4 of Definition \ref{def:B-ordering}. Therefore, $G$ must satisfy the following conditions:
        $$G \in B(I) \setminus B(I-1),\  \min(G) = \min(F) = 1, \text{ and }G' <_{B'} F',$$
        where $F' = \varphi_2(F\setminus\{1,2\})$, $G' = \varphi_2(G\setminus\{1,2\})$, and $B' = \varphi_2(B(I(\{1,2\}), 3)\setminus B((I-1)(\{1,2\}), 3))$.

        Since $G' <_{B'} F'$, the induction hypothesis gives $X_{B'}(F') \not\subseteq G'$. Therefore, we can see that

        $$X_B(F) = \varphi^{-1}_2(X_{B'}(F')) \not\subseteq \varphi^{-1}_2(G') = G\setminus \{1,2\}.$$

        Finally, since $X_B(F) \cap \{1,2\} = \emptyset$, we conclude that
        $$X_B(F) \not\subseteq G,$$
        as desired.
        
        \item \textbf{Suppose $\bm{F \in B(I-1)}$}. Let $G \neq F$ be a facet of $B$ containing $X_B(F)$. It is enough to show that $F <_B G$.

        Note that in this case $m=0$, and so $X_B(F) = (a_1, a_3, \cdots, a_{2t-1})$. By our assumption, $G \supseteq (a_1, a_3, \cdots, a_{2t-1})$, hence it follows that $G <_p F$. This means that $G$ is also in $B(I-1)$. Since $F$ and $G$ are both in $B(I-1)$ and $G <_p F$, Condition 2 of Definition \ref{def:B-ordering} implies that $F<_B G$.
        
        \item \textbf{Suppose $\bm{F \in B(I)\setminus B(I-1)}$ and $\bm{\min(F) > 1}$.} Let $G \neq F$ be a facet of $B$ containing $X_B(F)$. It is enough to show that $F <_B G$.

        If $G \in B(I-1)$, then $F <_B G$ by Condition 1 of Definition \ref{def:B-ordering}. On the other hand, if $G \notin B(I-1)$, then by Lemma \ref{lemma:condition-2} we only need to consider the following cases:

        \begin{enumerate}
            \item Suppose that $\min(F) < \min(G)$. Then $F <_B G$ by Condition 3 of Definition \ref{def:B-ordering}.
            
            \item Suppose that $\min(G) = \min(F) = k$. By Condition 4 of Definition \ref{def:B-ordering}, it is enough to show that $F' <_{B'} G'$, where
            \begin{align*}
                F' &\coloneq \varphi_{k+1}(F \setminus \{k,k+1\}) \\
                G' &\coloneq \varphi_{k+1}(G \setminus \{k,k+1\}) \\
                B' &\coloneq \varphi_{k+1}(B(I(\{k,k+1\}), k+2)\setminus B((I-1)(\{k,k+1\}), k+2)).
            \end{align*}
            By the inductive hypothesis, if $X_{B'}(F') \subseteq G'$, then $F' <_{B'} G'$.  Therefore, it is enough to show that $X_{B'}(F') \subseteq G'$. 

            It is easy to check that 
            $$X_{B'}(F') \subseteq \varphi_{k+1}(X_B(F)\setminus \{k,k+1\}).$$
            Therefore, 
            \begin{align*}
                X_B(F) \subseteq G &\implies X_B(F) \setminus \{k,k+1\} \subseteq G\setminus\{k,k+1\} \\
                &\implies \varphi_{k+1}(X_B(F) \setminus \{k,k+1\}) \subseteq \varphi_{k+1}(G\setminus\{k,k+1\}) \\
                &\implies X_{B'}(F') \subseteq G',
            \end{align*}
            as desired.
        \end{enumerate}
    \end{enumerate}
\end{proof}

We conclude the paper with the following remark.
\begin{remark}
The main result of this paper shows that relative squeezed balls are shellable. Since any simplicial complex of the form $B\coloneq B(I) \setminus B(J)$ is a proper full-dimensional subcomplex of 
the boundary complex of the cyclic polytope, $B$ is a pseudomanifold with boundary. Shellability of $B$ then implies that it is a ball. This provides an alternative proof of the result from \cite[Lemma 3.11]{novikManyNeighborlySpheres2024} that relative squeezed balls are indeed balls.

Furthermore, the description of $X_B(F)$ guarantees that all restriction faces of $B$ have size less than or equal to $t$. Consequently, 
$$h_i\bigl(B(I)\setminus B(J)\bigr)=0 \text{ for all } i>t,$$ and so $B(I) \setminus B(J)$ is $t$-stacked. In fact, in the case $J = I-1$, the value of $m$ from Definition \ref{def:X_B(F)} cannot be zero. Hence, in this situation all restriction faces have size strictly smaller than $t$. Thus,
\[
  h_i\bigl(B(I)\setminus B(I-1)\bigr)=0 \quad \text{for all } i\ge t,
\]
and so $B(I)\setminus B(I-1)$ is $(t-1)$-stacked. This yields an alternative proof of the second part of \cite[Lemma 3.11]{novikManyNeighborlySpheres2024}.
\end{remark}

\section*{Data Availability}
No datasets were generated or analysed during the current study.

\section*{Acknowledgments}
The author would like to express her sincere gratitude to her academic advisor, Isabella Novik, for introducing this problem and for her insightful guidance and continuous feedback during the course of this work. The author was partially supported by a graduate fellowship from NSF grant DMS-2246399.

\newpage
\appendix
\appendixpage
\section{Proofs of the shelling lemmas from Section 3}

\begin{lemma}[Prefix Replacement] \label{lemma:prefix-replacement} Let $\Delta$ be a shellable simplicial sphere or a shellable simplicial ball with a shelling order $F_1, \ldots, F_s$ and let $\Delta_1$ be the simplicial ball generated by an initial segment $F_1, \ldots, F_m$ of this shelling, where $1 \leq m \leq s$. Let $\Delta_2$ be a shellable simplicial ball with shelling $G_1, \ldots, G_r$. Suppose that $\partial \Delta_2 = \partial \Delta_1$ and that $\Delta' \coloneq \Delta_2 \cup (\Delta \setminus \Delta_1)$  is a simplicial ball or sphere in which no facet of $\Delta_2$ is a facet of $\Delta \setminus \Delta_1$. Then $\Delta'$ is shellable with shelling order $G_1, \ldots, G_r, F_{m+1}, \ldots, F_s$.
\end{lemma}

\begin{proof}
    Let $d-1$ denote the dimension of $\Delta$, so that the ridges of $\Delta$ have dimension $d-2$. It is enough to verify that $F_i$ satisfies the shelling condition on $\Delta'$ for $m < i \leq s$. That is, we want to show that 
    $$\overline{F_i} \cap (\Delta_2 \cup (\overline{F_{m+1}} \cup \cdots \cup \overline{F_{i-1}}))$$
    is pure of dimension $d-2$.
    This follows by considering the following equalities:
    \begin{align*}
        \overline{F_i} \cap (\Delta_2 \cup (\overline{F_{m+1}} \cup \cdots \cup \overline{F_{i-1}})) &= (\overline{F_i} \cap \Delta_2) \cup (\overline{F_i} \cap (\overline{F_{m+1}} \cup \cdots \cup \overline{F_{i-1}}))\\
        &= (\overline{F_i} \cap \partial\Delta_2) \cup (\overline{F_i} \cap (\overline{F_{m+1}} \cup \cdots \cup \overline{F_{i-1}}))\\
        &= (\overline{F_i} \cap \partial\Delta_1) \cup (\overline{F_i} \cap (\overline{F_{m+1}} \cup \cdots \cup \overline{F_{i-1}}))\\
        &= (\overline{F_i} \cap \Delta_1) \cup (\overline{F_i} \cap (\overline{F_{m+1}} \cup \cdots \cup \overline{F_{i-1}}))\\
        &= \overline{F_i} \cap (\Delta_1 \cup (\overline{F_{m+1}} \cup \cdots \cup \overline{F_{i-1}})).
    \end{align*}
    Here the second equality holds because $\Delta'$ is a simplicial ball or sphere and $F_i$ is a facet of $\Delta'$ that is not a facet of $\Delta_2$, so its intersection with $\Delta_2$ lies in $\partial\Delta_2$. The third equality follows from the assumption $\partial\Delta_2=\partial\Delta_1$. The fourth equality holds because $\Delta$ is a simplicial ball or sphere and $F_i$ is a facet of $\Delta$ that is not a facet of $\Delta_1$, so its intersection with $\Delta_1$ lies in $\partial\Delta_1$.
    
    Finally, we know that $\overline{F_i} \cap (\Delta_1 \cup (\overline{F_{m+1}} \cup \cdots \cup \overline{F_{i-1}}))$ is pure of dimension $d-2$ because this is the shelling condition for $F_i$ in the shelling $F_1, \ldots, F_s$ of $\Delta$. Hence the proposed ordering is a shelling of $(\Delta\setminus\Delta_1)\cup\Delta_2$.
    \end{proof}

\begin{lemma}[Reversibility\footnote{See Lemma 8.10 of \cite{zieglerLecturesPolytopes2007} for a version of this result.}]\label{lemma:reversibility}
    If $F_1, F_2, \cdots, F_s$ is a shelling order for a simplicial sphere $\Delta$, then so is the reverse order $F_s, F_{s-1}, \cdots, F_1$.
\end{lemma}

\begin{proof}
Let $d-1$ denote the dimension of $\Delta$, so that the ridges of $\Delta$ have dimension $d-2$.

Fix $1 \leq k \leq s$, and consider the complexes $\Delta_k = \bigcup_{i=1}^{k-1} \overline{F_i}$ and $\Delta'_k = \bigcup_{i=k+1}^s \overline{F_i}$.
We know that for $k > 1$, $\overline{F_k} \cap \Delta_k$ is pure of dimension $d-2$, and need to verify that, for $k < s$, so is $\overline{F_k} \cap \Delta'_{k}$.

Since $\Delta$ is a simplicial sphere, every ridge is contained in exactly two facets. Now consider any ridge $R \subset F_k$. It is contained in exactly one other facet, say $F_j$, which is either in $\Delta_k$ or in $\Delta'_k$.

Since $\partial \overline{F_k}$ is a sphere of dimension $d-2$, and its set of ridges splits into those shared with earlier facets and those shared with later facets, it follows that $\overline{F_k} \cap \Delta'_k = \partial \overline{F_k} \setminus (\overline{F_k} \cap \Delta_k)$. Furthermore, $\overline{F_k} \cap \Delta_k \neq \partial \overline{F_k}$ because $k < s$ and only the last facet of a shelling of a sphere gets glued along its entire boundary.  This implies that $\overline{F_k} \cap \Delta'_k$ is non-empty and, therefore, must be a pure $(d-2)$-dimensional subcomplex of $\partial \overline{F_k}$.

Thus, for each $k$, the intersection of $\overline{F_k}$ with the union of the facets coming after it in the reverse order is pure of dimension $d-2$, showing that $F_s, \dots, F_1$ is a shelling order.
\end{proof}

\section{Proof of the technical lemmas from Section 4}

\begin{lemma}\label{lemma:about-F-and-m}
    Let $F \coloneq (a_1, \ldots, a_{2t})$ be a facet of $B \coloneq B(I) \setminus B(J)$, and let $m = \min\{i : (a_1, \ldots, a_{2i})\cup (a_{2i+1}+1, \ldots, a_{2t}+1) \in B(I)\}$. Then, for $0\leq j < m$,
    $$F_{[2j+1, 2t]} \in B(I(F_{[1,2j]}), a_{2j}+1)\setminus B(I(F_{[1,2j]}) - 1, a_{2j}+1).$$
    Furthermore, for $j=m$,
    $$F_{[2j+1, 2t]} \in B(I(F_{[1,2j]}) - 1, a_{2j}+1).$$
\end{lemma}

\begin{proof}
    First, since $F \in B(I)$, it follows that 
    $$F_{[2j+1, 2t]} \in B(I(F_{[1,2j]}), a_{2j}+1).$$
    Now suppose that $0 \leq j < m$. By the minimality of $m$, we know that $F_{[1,2j]} \cup (F_{[2j+1,2t]}+1) \notin B(I)$. Equivalently, $F_{[2j+1,2t]}+1 \notin B(I(F_{[1,2j]}), a_{2j}+1)$, which implies that for $0 \leq j < m$,
    $$F_{[2j+1,2t]} \notin B(I(F_{[1,2j]})-1, a_{2j}+1),$$
    as desired.
    
    Now suppose that $j=m$. By the definition of $m$, $F_{[1,2m]} \cup (F_{[2m+1,2t]}+1) \in B(I)$. Equivalently, $F_{[2m+1,2t]}+1 \in B(I(F_{[1,2m]}), a_{2m}+1)$, which implies that for $j=m$,
    $$F_{[2m+1,2t]} \in B(I(F_{[1,2m]})-1, a_{2m}+1).$$
    This completes the proof.
\end{proof}

\begin{lemma}\label{lemma:G_case_2a}
    Let $F \coloneq (a_1, \ldots, a_{2t})$ be a facet of $B \coloneq B(I) \setminus B(J)$, let $x \in X_B(F)$, let $G \coloneq G_B(F, x)$ be defined according to case 2a of Definition \ref{def:G_B(F,x)}, and let $r = \max\{j : F_{[1,2j]}=G_{[1,2j]}\}$. Then, for $0 \leq j \leq r$,
    $$G_{[2j+1, 2t]} \in B(I(F_{[1,2j]}), a_{2j}+1)\setminus B(I(F_{[1,2j]}) - 1, a_{2j}+1).$$
\end{lemma}

\begin{proof}
    Let $0 \leq j \leq r$. Since $G <_p F$ and $F \in B(I)$, it follows that $G$ must be in $B(I)$. Since $G_{[1,2j]}=F_{[1,2j]}$, this implies that, for $0 \leq j \leq r$,
    $$G_{[2j+1, 2t]} \in B(I(G_{[1,2j]}), a_{2j}+1) = B(I(F_{[1,2j]}), a_{2j}+1).$$
    
    To show that $G_{[2j+1, 2t]} \not \in B(I(F_{[1,2j]})-1, a_{2j}+1)$, let 
        \[
            H = (a_1, \ldots, a_{2(m-1)}) \cup (a_{2m-1}+1, \ldots, a_{2t}+1).
        \]
    By the minimality of $m$, $H \not \in B(I)$. Since $H_{[1,2j]} = F_{[1,2j]}$, this implies that $H_{[2j+1, 2t]} \notin B(I(F_{[1,2j]}), a_{2j}+1)$. Consequently, 
    $$H_{[2j+1, 2t]}-1 \not \in B(I(F_{[1,2j]})-1, a_{2j}+1).$$ 
    Since $H_{[2j+1, 2t]}-1 <_p G_{[2j+1, 2t]}$, it follows that $G_{[2j+1, 2t]}$ is also not in $B(I(F_{[1,2j]})-1, a_{2j}+1)$. We conclude that, for $0 \leq j \leq r$, 
    $$G_{[2j+1, 2t]} \in B(I(F_{[1,2j]}), a_{2j}+1)\setminus B(I(F_{[1,2j]})-1, a_{2j}+1),$$
    as desired.    
\end{proof}

\begin{lemma}\label{lemma:G_case_2b}
    Let $F \coloneq (a_1, \ldots, a_{2t})$ be a facet of $B \coloneq B(I) \setminus B(J)$, let $x \in X_B(F)$, let $G \coloneq G_B(F, x)$ be defined according to case 2b of Definition \ref{def:G_B(F,x)}, and let $r = \max\{j : F_{[1,2j]}=G_{[1,2j]}\}$. Recall that $m = \min\{i : (a_1, \ldots, a_{2i})\cup (a_{2i+1}+1, \ldots, a_{2t}+1) \in B(I)\}$. Then, for $0 \leq j \leq r$,
    $$G_{[2j+1, 2t]} \in B(I(F_{[1,2j]}), a_{2j}+1)\setminus B(J(F_{[1,2j]}), a_{2j}+1).$$
    Furthermore, for $0 \leq j < m$,
    $$G_{[2j+1, 2t]} \in B(I(F_{[1,2j]}), a_{2j}+1)\setminus B(I(F_{[1,2j]})-1, a_{2j}+1).$$
\end{lemma}

\begin{proof}
    Let $H = (a_1, \ldots, a_{2m}) \cup (a_{2m+1}+1, \ldots, a_{2t}+1)$. Note that, by the definition of $m$, $H \in B(I)$. Since $G <_p H$, $G$ must also be in $B(I)$. Since for $0 \leq j \leq r$, $F_{[1,2j]}= G_{[1,2j]}$, we obtain that
    $$G_{[2j+1, 2t]} \in B(I(G_{[1,2j]}), a_{2j}+1) = B(I(F_{[1,2j]}), a_{2j}+1) \text{ for } 0 \leq j \leq r.$$
    
    Furthermore, since $F <_p G$ and $F \notin B(J)$, it follows that $G \notin B(J)$. Therefore,
    $$G_{[2j+1, 2t]} \notin B(J(G_{[1,2j]}), a_{2j}+1) = B(J(F_{[1,2j]}), a_{2j}+1) \text{ for } 0 \leq j \leq r.$$
    
    Now suppose that $0 \leq j < m$. By Lemma \ref{lemma:about-F-and-m}, $F_{[2j+1, 2t]} \notin B(I(F_{[1, 2j]})-1, a_{2j}+1)$. Since $F_{[2j+1, 2t]} <_p G_{[2j+1, 2t]}$, this implies that, for $0 \leq j < m$,
    $$G_{[2j+1, 2t]} \notin B(I(F_{[1,2j]})-1, a_{2j}+1).$$
    This completes the proof.
\end{proof}

\begin{lemma}\label{lemma:condition-2}
    Let $F\coloneq (a_1, \ldots, a_{2t})$ be a facet of $B$ such that $F \in B(I) \setminus B(I-1)$ and $\min(F) >1$, and let $G$ be a facet of $B$ containing $X_B(F)$. If $\min(G) < \min(F)$, then $G \in B(I-1)$.
\end{lemma}

\begin{proof}
    Since $m$ is such that 
    $$(a_1, \ldots, a_{2m}) \cup (a_{2m+1}+1, \ldots, a_{2t}+1) \in B(I),$$ 
    it is enough to show that
    $$G \leq_p (a_1 -1, \ldots, a_{2m}-1) \cup (a_{2m+1}, \ldots, a_{2t}).$$

    Since $\min(G) < \min(F)$ we have $b_1 \leq a_1 - 1$, and therefore
    \begin{equation} \label{eq:i=0}
    (b_1, b_2) \leq_p (a_1-1, a_2-1).
    \end{equation}
    Furthermore, note that since $F \notin B(I-1)$ and $\min(F)>1$, the set $X_B(F)$ has $t-1$ elements, $(x_1, \ldots, x_{t-1})$, and they are pairwise non-consecutive integers. Note also that $G \coloneq (b_1, \ldots, b_{2t})$ has $t$ blocks $(b_{2i+1}, b_{2i+2})$ of consecutive integers, for $0 \leq i \leq t-1$. Since the $x_i$ are pairwise non-consecutive, they lie in $t-1$ distinct blocks of $G$. By Definition \ref{def:X_B(F)}, every element of $X_B(F)$ is at least $a_2$, and so none of them lie in the first block, $(b_1, b_2)$. This forces $x_i \in (b_{2i+1},b_{2i+2})$ for $1 \leq i \leq t-1$, which heavily restricts what $G$ can be. In particular, we have
    $$(b_{2i+1}, b_{2i+2}) \leq_p (x_i, x_i+1) \quad \text{for } 1 \leq i \leq t-1.$$
    Since $x_i = a_{2i}$ for $1 \leq i < m$, it follows that 
    \begin{equation} \label{eq:1-leq-i-lt-m}
    (b_{2i+1}, b_{2i+2}) \leq_p (a_{2i}, a_{2i}+1) \leq_p (a_{2i+1}-1, a_{2i+2}-1) \quad \text{for } 1 \leq i < m.
    \end{equation}
    On the other hand, since $x_i = a_{2i+1}$ for $m \leq i < t$, we see that
    \begin{equation} \label{eq:m-leq-i-lt-t}
    (b_{2i+1}, b_{2i+2}) \leq_p (a_{2i+1}, a_{2i+1}+1) = (a_{2i+1}, a_{2i+2}) \quad \text{for } m \leq i < t.
    \end{equation}
    Putting equations \ref{eq:i=0}, \ref{eq:1-leq-i-lt-m}, and \ref{eq:m-leq-i-lt-t} together, we conclude that
    $$G \leq_p (a_1 -1, \ldots, a_{2m}-1) \cup (a_{2m+1}, \ldots, a_{2t}).$$
    This completes the proof.
\end{proof}

\newpage
\printbibliography

\end{document}